\documentclass[10pt]{amsart}
\usepackage{amsmath,amssymb,latexsym,hyperref,url,graphicx}

\graphicspath{{figures/}}

\newtheorem{theorem}{Theorem}
\newtheorem{conjecture}[theorem]{Conjecture}
\newtheorem{proposition}[theorem]{Proposition}

\newcommand{\ihat}{\hat\imath}
\newcommand{\jhat}{\hat\jmath}
\newcommand{\khat}{\hat{k}}
\newcommand{\R}{\mathbb{R}}
\newcommand{\nequiv}{\mathrel{\not\equiv}}
\newcommand{\Mod}{\textnormal{mod}}
\newcommand{\ParseWords}{\textnormal{ParseWords}}
\newcommand{\LeftCombTree}{\textnormal{Left}\-\textnormal{Comb}\-\textnormal{Tree}}
\newcommand{\RightCombTree}{\textnormal{Right}\-\textnormal{Comb}\-\textnormal{Tree}}
\newcommand{\LeftCrookedTree}{\textnormal{Left}\-\textnormal{Crooked}\-\textnormal{Tree}}
\newcommand{\RightCrookedTree}{\textnormal{Right}\-\textnormal{Crooked}\-\textnormal{Tree}}
\newcommand{\LeftTurnTree}{\textnormal{Left}\-\textnormal{Turn}\-\textnormal{Tree}}
\newcommand{\RightTurnTree}{\textnormal{Right}\-\textnormal{Turn}\-\textnormal{Tree}}

\newcommand{\vcentergraphics}[1]{\ensuremath{\vcenter{\hbox{\includegraphics{#1}}}}}
\newcommand{\smalltree}[2]{\ensuremath{\vcenter{\hbox{\scalebox{.7}{\includegraphics{binarytree#1-#2}}}}}}

\begin{document}

\title{Toward a language theoretic proof \\ of the four color theorem}

\author{Bobbe Cooper}
\address{
	School of Mathematics \\
	206 Church St.\ S.E. \\
	Minneapolis, MN 55455, USA
}
\author{Eric Rowland}
\address{
	School of Computer Science \\
	University of Waterloo \\
	Waterloo, ON N2L 3G1, Canada
}
\author{Doron Zeilberger}
\address{
	Department of Mathematics \\
	Rutgers University \\
	Piscataway, NJ 08854, USA
}

\date{November 22, 2011}

\thanks{The second author was supported in part by NSF grant DMS-0239996; the third author was supported in part by NSF grant DMS-0901226.}

\begin{abstract}
This paper considers the problem of showing that every pair of binary trees with the same number of leaves parses a common word under a certain simple grammar.  We enumerate the common parse words for several infinite families of tree pairs and discuss several ways to reduce the problem of finding a parse word for a pair of trees to that for a smaller pair.  The statement that every pair of trees has a common parse word is equivalent to the statement that every planar graph is four-colorable, so the results are a step toward a language theoretic proof of the four color theorem.
\end{abstract}

\maketitle
\markboth{Cooper, Rowland, and Zeilberger}{Toward a language theoretic proof of the four color theorem}

\section{Introduction}\label{introduction}

Let $G$ be the context-free grammar with start symbols $0, 1, 2$ and formation rules $0 \to 12$, $0 \to 21$, $1 \to 02$, $1 \to 20$, $2 \to 01$, $2 \to 10$.  An $n$-leaf tree $T$ \emph{parses} a length-$n$ word $w$ on $\{0, 1, 2\}$ if $T$ is a valid derivation tree for $w$ under the grammar $G$; that is, there is a labeling of the vertices of $T$ compatible with the formation rules such that the leaves of $T$, from left to right, are labeled with the letters of $w$.  For example, the tree
\[
	\vcentergraphics{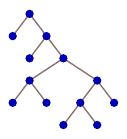}
\]
parses the word $0110212$, as can be seen by this labeling:
\[
	\vcentergraphics{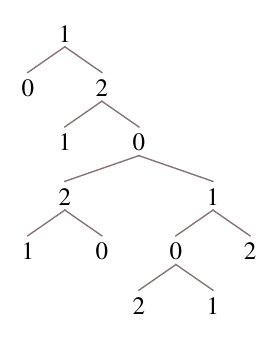}
\]
Note that labeling the leaves of a tree uniquely determines a labeling of the internal vertices under $G$ if a valid labeling exists.

Certainly $G$ is \emph{ambiguous} --- there exist distinct trees that parse the same word; for example, the trees
\[
	\vcentergraphics{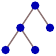} \qquad \vcentergraphics{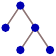}
\]
both parse $010$.  To take a somewhat larger example, the trees
\[
	\vcentergraphics{binarytree7-7} \qquad \vcentergraphics{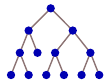}
\]
both parse the word $0110212$:
\[
	\vcentergraphics{binarytree7-7labeled} \qquad \vcentergraphics{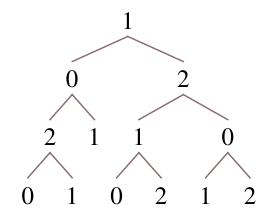}
\]

However, something much stronger can be said about this grammar.

\begin{theorem}\label{main}
The grammar $G$ is totally ambiguous.
\end{theorem}

That is, every pair of derivation trees with the same number of leaves has at least one word that they both parse.  Kauffman~\cite{Kauffman} proved this theorem (in a slightly different form, as we describe below) by showing that it is equivalent to the four color theorem --- the statement that every planar graph is four-colorable.  The four color theorem was proved by Appel, Haken, and Koch~\cite{Appel--Haken, Appel--Haken--Koch} using substantial computing resources.  The hope of the present authors is that a direct proof of Theorem~\ref{main} will be shorter than the known proofs of the four color theorem, thereby providing a shorter proof of the four color theorem.

In this paper we describe first results in this direction.
First we show that Theorem~\ref{main} is equivalent to Kauffman's formulation.
Section~\ref{parameterized families} determines explicit common parse words for several simple parameterized families of tree pairs.
In Section~\ref{general families} we establish existence of parse words for more general families.
In Section~\ref{turn trees} we enumerate the common parse words of a $3$-parameter family of tree pairs.
We conclude in Section~\ref{reducing} by discussing in more generality methods of reducing the problem of finding a common parse word for a pair of trees.

A \emph{Mathematica} package~\cite{ParseWords} and a Maple package~\cite{LOU} that accompany this paper and facilitate the discovery of the results we present can be downloaded from the respective web sites of the second and third authors.

\section{Relationship to the cross product}\label{equivalence}

The set of possible derivation trees under $G$ is the set of \emph{binary trees} --- trees in which each vertex has either $0$ or $2$ children.  (All trees in the paper are rooted and ordered.)

Let $|w|$ be the length of the word $w$, and let $|w|_i$ be the number of occurrences of the letter $i$ in $w$.

\begin{proposition}\label{root label}
Let $w$ be a word of length $n$ on $\{0, 1, 2\}$ and $T$ an $n$-leaf binary tree that parses $w$.  Then for some permutation $(r,s,t)$ of $(0,1,2)$,
\[
	|w|_s \equiv |w|_t \nequiv |w|_r \equiv |w| \mod 2.
\]
Moreover, the root of $T$ receives the label $r$ when parsing $w$.
\end{proposition}

\begin{proof}
The congruence holds for the three words of length $1$, and the derivation rules of $G$ preserve it because all four terms change parity with each rule application.
\end{proof}

It follows that if the parities of $|w|_0$, $|w|_1$, and $|w|_2$ are equal then no tree parses $w$ under $G$.  If on the other hand the parity of $|w|_r$ differs from the other two, then $r$ is an invariant of $w$ in the sense that any tree parsing $w$ has its root labeled $r$.

Kauffman~\cite{Kauffman} formulated Theorem~\ref{main} not in terms of a grammar but in terms of the cross product on the standard unit vectors $\ihat, \jhat, \khat$ in $\R^3$.
The cross product satisfies
\begin{align*}
	\ihat \times \jhat &= (-\ihat) \times (-\jhat) = (-\jhat) \times \ihat = \jhat \times (-\ihat) = \khat, \\
	\jhat \times \khat &= (-\jhat) \times (-\khat) = (-\khat) \times \jhat = \khat \times (-\jhat) = \ihat, \\
	\khat \times \ihat &= (-\khat) \times (-\ihat) = (-\ihat) \times \khat = \ihat \times (-\khat) = \jhat.
\end{align*}
Further, for every vector $v \in \R^3$ we have
\begin{align*}
	v \times v &= v \times (-v) = 0, \\
	v \times 0 &= 0 \times v = 0.
\end{align*}
The cross product on $\R^3$ is not associative, so in general the expression $v_1 \times v_2 \times \cdots \times v_n$ is ambiguous; to evaluate it for a given tuple $(v_1, v_2, \dots, v_n)$ we must choose an order in which to compute the $n-1$ cross products.
Let us call such an order an \emph{$n$-bracketing}.
Kauffman showed that the four color theorem is equivalent to the statement that for every pair of $n$-bracketings there exists an $n$-tuple $(v_1, v_2, \dots, v_n) \in \{\ihat, \jhat, \khat\}^n$ such that the two bracketings of $v_1 \times v_2 \times \cdots \times v_n$ evaluate to the same nonzero vector.

We now develop tools to show that Theorem~\ref{main} is equivalent to Kauffman's statement.
Roughly speaking, we show that we can replace $\pm \ihat \to 0$, $\pm \jhat \to 1$, and $\pm \khat \to 2$.
It is easy to see that $n$-bracketings are in bijection with $n$-leaf binary trees.
Therefore, given an $n$-bracketing of $v_1 \times v_2 \times \cdots \times v_n$, one may label each internal vertex of the corresponding binary tree with the cross product of the labels of its children (in order).
The condition that the bracketing does not evaluate to $0$ is equivalent to the condition that the evaluation does not encounter the product $v \times v$ or $v \times (-v)$, hence our formation rules for the grammar $G$.

Each $(v_1, v_2, \dots, v_n) \in \{\ihat, \jhat, \khat\}^n$ possesses an invariant analogous to that of Proposition~\ref{root label}.
To see what this invariant is, consider the quaternion group, whose elements are $Q = \{1, i, j, k, -1, -i, -j, -k\}$ and whose binary operation $\cdot$ satisfies
\begin{align*}
	i \cdot j &= (-i) \cdot (-j) = (-j) \cdot i = j \cdot (-i) = k, \\
	j \cdot k &= (-j) \cdot (-k) = (-k) \cdot j = k \cdot (-j) = i, \\
	k \cdot i &= (-k) \cdot (-i) = (-i) \cdot k = i \cdot (-k) = j,
\end{align*}
as well as identities such as $(-1) \cdot (-1) = 1$ and $(-1) \cdot i = -i$ suggested by the notation.
Further, for $q \in \{i, j, k, -i, -j, -k\}$ we have $q \cdot q = -1$ and $q \cdot (-q) = 1$.
Consider $\phi : \{\ihat, \jhat, \khat, -\ihat, -\jhat, -\khat\} \to \{i, j, k, -i, -j, -k\}$ mapping
\begin{align*}
	\phi(\ihat) &= i, \qquad \phi(-\ihat) = -i, \\
	\phi(\jhat) &= j, \qquad \phi(-\jhat) = -j, \\
	\phi(\khat) &= k, \qquad \phi(-\khat) = -k.
\end{align*}
The map $\phi$ is a ``partial homomorphism'' in the sense that $\phi(v_1 \times v_2) = \phi(v_1) \cdot \phi(v_2)$ if $v_1 \neq v_2$ and $v_1 \neq -v_2$.
This property allows us to establish the following invariant.

\begin{proposition}\label{root vector}
Let $(v_1, v_2, \dots, v_n) \in \{\ihat, \jhat, \khat\}^n$, and choose a bracketing of $v_1 \times v_2 \times \cdots \times v_n$ that does not evaluate to the zero vector.
Then this bracketing evaluates to $\phi^{-1}(\phi(v_1) \cdot \phi(v_2) \cdots \phi(v_n))$.
\end{proposition}

\begin{proof}
Since the bracketing of $v_1 \times v_2 \times \cdots \times v_n$ does not evaluate to $0$, each of the $n-1$ cross products is an operation on two linearly independent vectors.
Therefore we may emulate the evaluation of the bracketing in $Q$ rather than in $\R^3$, because replacing $\times$ with $\cdot$ is consistent with $\phi$.
Since $\cdot$ is associative, the bracketing evaluates to $\phi(v_1) \cdot \phi(v_2) \cdots \phi(v_n)$ in $Q$.
Moreover, since we do not encounter $0$ in $\R^3$, we do not encounter $-1$ or $1$ in $Q$; in particular, $\phi(v_1) \cdot \phi(v_2) \cdots \phi(v_n) \in \{i, j, k, -i, -j, -k\}$, and therefore $\phi^{-1}(\phi(v_1) \cdot \phi(v_2) \cdots \phi(v_n))$ exists.
\end{proof}

A result of Proposition~\ref{root vector} is that we can drop the condition in Kauffman's statement that the two bracketings evaluate to the same vector.
Therefore the four color theorem is equivalent to the statement that for every pair of $n$-bracketings there exists an $n$-tuple $(v_1, v_2, \dots, v_n) \in \{\ihat, \jhat, \khat\}^n$ such that the two bracketings of $v_1 \times v_2 \times \cdots \times v_n$ evaluate to nonzero vectors.

Consider the homomorphism $\sigma : Q \to Q/\{1, -1\} \cong V$, where $V = \{e, 0, 1, 2\}$ is the Klein four-group, $e = \{1, -1\}$ is the identity element, $0 = \{i, -i\}$, $1 = \{j, -j\}$, and $2 = \{k, -k\}$.
Let $\tau : \{\ihat, \jhat, \khat\} \to \{0, 1, 2\}$ be defined by $\tau(v) = \sigma(\phi(v))$.
In other words, $\tau$ removes the hat and forgets the sign.
Since $\sigma$ is a homomorphism, evaluating $T_1$ and $T_2$ at $(v_1, v_2, \dots, v_n) \in \{\ihat, \jhat, \khat\}^n$ results in nonzero vectors if and only if $\tau(v_1) \tau(v_2) \cdots \tau(v_n)$ is a parse word for $T_1$ and $T_2$.
Therefore, for $n$-leaf binary trees $T_1$ and $T_2$, the $n$-tuples $(v_1, v_2, \dots, v_n) \in \{\ihat, \jhat, \khat\}^n$ that evaluate to nonzero vectors when bracketed by $T_1$ and $T_2$ are in bijection with words $w \in \{0, 1, 2\}^n$ that are parsed by both $T_1$ and $T_2$.
It follows that Theorem~\ref{main} is equivalent to the four color theorem.

\section{Parameterized families}\label{parameterized families}

In this section we introduce several families of binary trees and enumerate the parse words of several pairs of these trees.
First we establish some additional terminology.

If $T$ parses a word $w$ on $\{0, 1, 2\}$, then $T$ also parses all words obtained from $w$ by permuting the letters in the alphabet.  Let $\ParseWords(T_1, T_2)$ be the set of equivalence classes (under permutations) of words parsed by both trees $T_1$ and $T_2$.  We abuse notation slightly by writing a representative of each equivalence class.  For example, it turns out that for the pair of $7$-leaf trees mentioned in Section~\ref{introduction} there is only one equivalence class of parse words, so for those trees we write
\[
	\ParseWords(T_1, T_2) = \{0110212\}.
\]
Often we will take this representative to be the word in the equivalence class which is lexicographically first --- words of the form $0$ or $0^k 1 v$.
However, we will depart from this convention when convenient.
The four color theorem is equivalent to the statement that for every pair of $n$-leaf binary trees $T_1$ and $T_2$ we have $\ParseWords(T_1,T_2) \neq \{\}$.

The \emph{level} of a vertex is its distance from the root.  That is, the root lies on level $0$, the root's children lie on level $1$, and so on.

A \emph{path tree} is a binary tree with at most two vertices in each level.  The $5$-leaf path trees are as follows.
\[
	\includegraphics{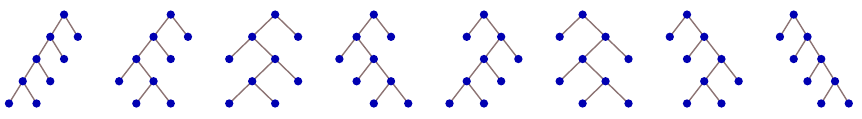}
\]
The two leaves on level $n-1$ in an $n$-leaf path tree are called the \emph{bottom leaves}.

The set of $n$-leaf path trees is in trivial bijection to the set $\{l, r\}^{n-2}$ of $(n-2)$-length words on $\{l, r\}$:  Since each level has at most two vertices, at most one vertex in each level has children, so we may form a word that records which child --- left or right --- has children at each level.  We shall use this bijection to define several families of trees.

Because of their linear structure, path trees are simpler to work with than binary trees in general, so the emphasis of this paper is on path trees.  Indeed, several infinite families of pairs of path trees can be shown to satisfy Theorem~\ref{main} directly and have only a few parse words.  We take up this task now.  Some of the proofs work by finding out where the local conditions imposed by the two trees force a unique labeling and then just working out the consequences, so in some cases it may be quicker to prove the theorem for yourself than to read the proof provided.

Let $\LeftCombTree(n)$ be the $n$-leaf path tree corresponding to the word $l^{n-2}$.  The left comb trees for $n \geq 2$ are pictured below.
\[
	\includegraphics{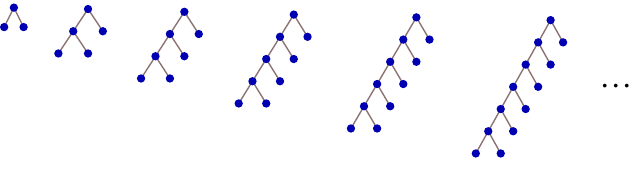}
\]
Let $\RightCombTree(n)$ be the $n$-leaf path tree corresponding to $r^{n-2}$; $\RightCombTree(n)$ is the left--right reflection of $\LeftCombTree(n)$.  We warm up with some \emph{comb}i\-na\-tor\-ics.

\begin{theorem}\label{comb-comb}
$\ParseWords(\LeftCombTree(n), \RightCombTree(n)) =$
\[
	\begin{cases}
		\left\{01^{n-2}2\right\}	& \text{if $n \geq 2$ is even} \\
		\left\{01^{n-2}0\right\}	& \text{if $n \geq 3$ is odd.}
	\end{cases}
\]
\end{theorem}

\begin{proof}
We build a common parse word from left to right --- up $\LeftCombTree(n)$ and down $\RightCombTree(n)$.  At every leaf, each tree will eliminate one possible label, so the parse word will turn out to exist and be unique.

The case $n=2$ can be established by testing all words of length $2$, so let $n \geq 3$.  Without loss of generality we may label the first two leaves $0$ and $1$.  It follows from this that the root of $\RightCombTree(n)$ receives the label $1$, the non-leaf (internal) vertex on the second level of $\RightCombTree(n)$ receives $2$, and therefore the internal vertex on the third level of $\RightCombTree(n)$ receives $0$.  This implies (from the right comb) that the third leaf cannot receive $0$.  However, from the left comb we find that the third leaf cannot receive $2$.  Therefore the third leaf receives $1$.  For the fourth leaf, the right comb precludes $2$ and the left comb precludes $0$, so the fourth leaf receives $1$.  Likewise all the way down the word through leaf $n-1$.  The internal vertex labels in each tree alternate between $0$ and $2$, except for the root which receives 1.  If $n$ is odd then the lowest internal vertex in the right comb receives $2$, so that the last leaf receives $0$; if $n$ is even then this internal vertex receives $0$, and the last leaf receives $2$.
\end{proof}

Note from the proof of this theorem that the internal labels corresponding to a common parse word of $\LeftCombTree(n)$ and $\RightCombTree(n)$ will match (top to bottom) if $n$ is odd, and will differ by the permutation which swaps $0$ and $2$ if $n$ is even.

Let $\LeftTurnTree(m,n)$ be the $(m+n)$-leaf path tree corresponding to $l^m r^{n-2}$, and let $\RightTurnTree(m,n)$ be the tree corresponding to $r^m l^{n-2}$.  Each of these trees is formed by ``gluing'' together two comb trees.  For example,
\[
	\LeftTurnTree(2,3) = \vcentergraphics{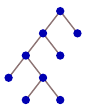}.
\]
The following theorem is a special case of the general treatment of two turn trees given in Section~\ref{turn trees}.

\begin{theorem}\label{turn-turn}
For $m \geq 1$,
\begin{multline*}
	\ParseWords(\LeftTurnTree(m, n), \RightTurnTree(1, m + n - 1)) = \\
	\begin{cases}
		\left\{0 0 1^{n-3} 2 0^m, 0 2 1^{n-3} 0 0^m\right\}	& \text{if $n \geq 3$ is odd} \\
		\left\{0 2 1^{n-3} 2 0^m, 0 0 1^{n-3} 0 0^m\right\}	& \text{if $n \geq 4$ is even.}
	\end{cases}
\end{multline*}
\end{theorem}

\begin{proof}
Without loss of generality, label the last leaf of each tree $0$.  The roots of the trees receive the same label, and thus the respective parents of the last leaf of each tree must receive $1$ and $2$ in some order, and the first leaf must be labeled $0$.  This implies that the last $m$ leaves are labeled $0$.  There are (up to permutation of $1$ and $2$) three possible options for the labels of leaves $n-1$ and $n$ (the bottom leaves of $\LeftTurnTree(m, n)$), namely $12$, $10$, and $01$.  Each of the first two options can be seen to yield a unique common parse word as given in the statement of the theorem.  The third option, in which leaves $n-1$ and $n$ are labeled $01$, is not valid, since then the sibling of leaf $n+1$ in $\RightTurnTree(1, m + n - 1)$ is labeled $0$, which contradicts leaf $n+1$ receiving $0$.
\end{proof}

If $w = w_1 w_2 \cdots w_m$ is a word of length $m$ and $x$ is a rational number whose denominator (in lowest terms) divides $m$, let
\[
	w^x = w^{\lfloor x \rfloor} w_1 w_2 \cdots w_{m \cdot (x - \lfloor x \rfloor)}
\]
be the word consisting of repeated copies of $w$ truncated at $m x$ letters.  For example, $(lr)^{7/2} = lrlrlrl$.

Let $\LeftCrookedTree(n)$ be the path tree corresponding to $(lr)^{(n-2)/2}$.  The left crooked trees for $n \geq 2$ are as follows.
\[
	\includegraphics{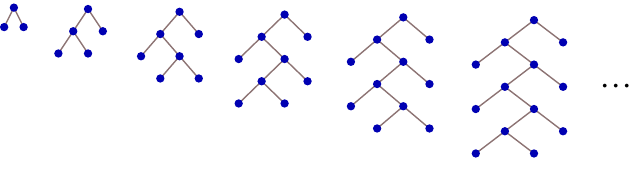}
\]
Let $\RightCrookedTree(n)$ be the path tree corresponding to $(rl)^{(n-2)/2}$ --- the left--right reflection of $\LeftCrookedTree(n)$.

The next two results determine the common parse words of a comb tree and the completely crooked trees of the same size.  Let $w^R$ be the left--right reversal of the word $w$.  Let $\Mod(n,3)$ be the smallest nonnegative integer congruent to $n$ modulo $3$.

\begin{theorem}\label{comb-crooked}
$\ParseWords(\LeftCombTree(n), \RightCrookedTree(n)) =$
\[
	\begin{cases}
		\left\{ \Mod(1-n,3) \left((012)^{n/6}\right)^R (012)^{(n-2)/6} \right\}		& \text{if $n \geq 2$ is even} \\
		\left\{ \Mod(1-n,3) \left((012)^{(n-3)/6}\right)^R (012)^{(n+1)/6} \right\}	& \text{if $n \geq 3$ is odd.}
	\end{cases}
\]
\end{theorem}

\begin{proof}
One checks that for $n=2$ the set of equivalence classes of parse words is $\{20\}$.

Inductively, assume that $\LeftCombTree(n-1)$ and $\RightCrookedTree(n-1)$ parse the word claimed and that this is the only word they both parse (up to permutations of the alphabet).  For even $n-1$, the two bottom leaves of $\RightCrookedTree(n-1)$ are leaves $\frac{n-1}{2}$ and $\frac{n+1}{2}$.  For odd $n-1$, they are leaves $\frac{n}{2}$ and $\frac{n+2}{2}$.  Observe that for even $n-1$ the right bottom leaf of $\RightCrookedTree(n-1)$ receives $0$, and for odd $n-1$ the left bottom leaf of $\RightCrookedTree(n-1)$ receives $0$.  For $n-1 \geq 4$ these are respectively the first and second of the two consecutive $0$s in the parse word.

We attach $\smalltree{2}{1}$ at the bottom of $\RightCrookedTree(n-1)$ to form $\RightCrookedTree(n)$ and insert $\smalltree{2}{1}$ at the corresponding place in $\LeftCombTree(n-1)$ to form $\LeftCombTree(n)$.  Label the new bottom leaves of $\RightCrookedTree(n)$ $12$ if $n-1$ is even and $21$ if $n-1$ is odd; we can label the corresponding leaves of $\LeftCombTree(n)$ the same by labeling their respective neighboring internal vertices $0$ and $1$ if $n-1$ is even and $0$ and $2$ if $n-1$ is odd.  The permutation $0 \to 2, 1 \to 0, 2 \to 1$ puts the new word in the form given in the theorem.

This process is reversible, so every parse word for $n$ comes from a parse word for $n-1$.
\end{proof}

The next theorem follows immediately from the previous theorem by labeling $\LeftCombTree(n)$ and $\RightCrookedTree(n)$ with a common parse word and then attaching the root of each tree as the left leaf of $\smalltree{2}{1}$.

\begin{theorem}\label{comb-crooked2}
$\ParseWords(\LeftCombTree(n), \LeftCrookedTree(n)) =$
\[
	\begin{cases}
		\begin{array}{l}
			\Big\{ \Mod(2-n,3) \left((012)^{(n-1)/6}\right)^R (012)^{(n-3)/6} \Mod(2-n,3), \\
			\phantom{\Big\{} \Mod(2-n,3) \left((012)^{(n-1)/6}\right)^R (012)^{(n-3)/6} \Mod(-n,3) \Big\}
		\end{array}	& \text{if $n \geq 3$ is odd} \\
		\begin{array}{l}
			\Big\{ \Mod(2-n,3) \left((012)^{(n-4)/6}\right)^R (012)^{n/6} \Mod(2-n,3), \\
			\phantom{\Big\{} \Mod(2-n,3) \left((012)^{(n-4)/6}\right)^R (012)^{n/6} \Mod(-n,3) \Big\}
		\end{array}	& \text{if $n \geq 4$ is even.}
	\end{cases}
\]
\end{theorem}

\begin{theorem}\label{crooked-crooked count}
For $n \geq 2$,
\[
	|\ParseWords(\LeftCrookedTree(n), \RightCrookedTree(n))| = 2^{\lfloor n/2 \rfloor - 1}.
\]
\end{theorem}

\begin{proof}
The cases $n = 2$ and $n = 3$ are easily verified.  In particular, every parse word of the $3$-leaf pair consisting of
\[
	\vcentergraphics{binarytree3-2} \qquad \vcentergraphics{binarytree3-1}
\]
is of the form $aba$ for $a \neq b$, where the two roots also get labeled $b$.

Let $n$ be odd.  Consider inductively extending $\LeftCrookedTree(n - 2)$ and $\RightCrookedTree(n - 2)$ by
\[
	\vcentergraphics{binarytree3-2} \qquad \vcentergraphics{binarytree3-1}
\]
respectively to obtain $\LeftCrookedTree(n)$ and $\RightCrookedTree(n)$.  Because the three new leaves are leaves $(n-1)/2$, $(n+1)/2$, and $(n+3)/2$ in both $\LeftCrookedTree(n)$ and $\RightCrookedTree(n)$, every parse word
\[
	w_1 w_2 \cdots w_{(n-3)/2} b w_{(n+1)/2} \cdots w_{n-3} w_{n-2}
\]
for the two $(n-2)$-leaf crooked trees can be extended to a parse word
\[
	w_1 w_2 \cdots w_{(n-3)/2} a b a w_{(n+1)/2} \cdots w_{n-3} w_{n-2}
\]
for the two $n$-leaf crooked trees.  Moreover, every parse word for the two $n$-leaf crooked trees can be obtained in this way.  Since there are two choices for $a$, there are twice as many parse words for the $n$-leaf crooked trees as for the $(n-2)$-leaf crooked trees, which establishes the statement for odd $n$; we see that $w = w_1 w_2 \cdots w_{n-1} w_n$ is a parse word for $\LeftCrookedTree(n)$ and $\RightCrookedTree(n)$ if and only if $w_i = w_{n+1-i} \neq w_{(n+1)/2}$ for $1 \leq i \leq \frac{n-1}{2}$.

Let $n$ be even.  Then every parse word of the $n$-leaf crooked trees can be obtained by extending a parse word $w_1 w_2 \cdots w_{n/2 - 1} b w_{n/2 + 1} \cdots w_{n-2} w_{n-1}$ of the $(n-1)$-leaf crooked trees.  Every parse word of $\smalltree{2}{1}$ in which the root receives label $b$ is of the form $ac$, where $a \neq b$ and $c \neq b$, so there are twice as many parse words for the $n$-leaf crooked trees as for the $(n-1)$-leaf crooked trees, which establishes the statement for even $n$.  Specifically, $w = w_1 w_2 \cdots w_{n-1} w_n$ is a parse word for $\LeftCrookedTree(n)$ and $\RightCrookedTree(n)$ if and only if $w_{n/2} \neq w_{n/2 + 1}$ and for some $b \in \{0, 1, 2\}$ we have $w_i = w_{n+1-i} \neq b$ for $1 \leq i \leq \frac{n}{2} - 1$.
\end{proof}

\section{General families}\label{general families}

Presumably explicit parse words can be found for various other parameterized families of tree pairs, but we now take a more general approach and establish results for tree pairs in which at least one of the trees does not come from a simple parameterized family.  Some of these results will be used in Section~\ref{turn trees}.  Note that, where stated, these results apply to not just path trees but binary trees in general.

\begin{proposition}\label{shared bottom leaf}
Let $n \geq 3$.  If the $i$th leaf is a bottom leaf in two $n$-leaf path trees, then the trees both parse the word $0^{k-1} 1 0^{n-k}$ for some $2 \leq k \leq n - 1$.
\end{proposition}

For example, this proposition applies to the pair in Theorem~\ref{crooked-crooked count} consisting of $\LeftCrookedTree(n)$ and $\RightCrookedTree(n)$.

\begin{proof}
If $i = 1$ then the second leaf is also a bottom leaf in both trees, so let $k = 2$; similarly, if $i = n$, let $k = n-1$.  If $2 \leq i \leq n-1$, let $k = i$.  Labeling the $k$th leaf $1$ and all other leaves $0$ produces a valid labeling of both trees because the internal vertices of the two trees on each level receive the same label, namely alternating between $2$ and $1$.
\end{proof}

We now give two propositions regarding extending a pair of binary trees by $\smalltree{2}{1}$.

\begin{proposition}\label{bottom-bottom}
Suppose $T_1'$ and $T_2'$ are $n$-leaf binary trees with a common parse word.  Extend each tree by attaching $\smalltree{2}{1}$ to leaf $i$, obtaining $T_1$ and $T_2$ respectively.  Then
\[
	|\ParseWords(T_1, T_2)| = 2 |\ParseWords(T_1', T_2')|.
\]
In particular, $T_1$ and $T_2$ have a common parse word.
\end{proposition}

\begin{proof}
Let $w$ be a parse word of $T_1'$ and $T_2'$.  Without loss of generality we may assume that $w_i = 0$.  Replacing $w_i$ by $12$ or $21$ produces a word that both $T_1$ and $T_2$ parse, and every parse word for the pair arises uniquely in this way.
\end{proof}

In the next proposition we consider extending a tree $T$ by inserting $\smalltree{2}{1}$ into the tree at an internal vertex to ``duplicate'' a leaf.  Fix $i$, and let $S$ be the tree hanging from the sibling vertex of leaf $i$.  Remove $S$ from its position, attach $\smalltree{2}{1}$ to the sibling of leaf $i$, and then reattach $S$ to a leaf of the new $\smalltree{2}{1}$ as follows.  If leaf $i$ is a left leaf, attach $S$ to the right leaf of the new $\smalltree{2}{1}$; if leaf $i$ is a right leaf, attach $S$ to the left leaf.  Therefore if leaf $i$ in $T$ is a left leaf, then leaves $i$ and $i+1$ in the extended tree are both left leaves, and if leaf $i$ in $T$ is a right leaf, then leaves $i$ and $i+1$ in the extended tree are right leaves.  We refer to this operation as \emph{duplicating} leaf $i$.

\begin{proposition}\label{bottom-comb}
Suppose $T_1'$ and $T_2'$ are $n$-leaf binary trees with a common parse word.  Extend $T_1'$ by attaching $\smalltree{2}{1}$ to leaf $i$, obtaining $T_1$.  Extend $T_2'$ to obtain $T_2$ by duplicating leaf $i$.  Then
\[
	|\ParseWords(T_1, T_2)| = |\ParseWords(T_1', T_2')|.
\]
In particular, $T_1$ and $T_2$ have a common parse word.
\end{proposition}

\begin{proof}
Let $w$ be a parse word of $T_1'$ and $T_2'$.  Without loss of generality we may assume that $w_i = 0$ and that the parent of leaf $i$ in $T_2'$ receives the label $1$.

If leaf $i$ is a left leaf in $T_2'$, then $T_2$ parses the word obtained by replacing $w_i$ by $21$ since duplicating leaf $i$ in $T_2'$ has the effect of the replacement
\[
	\vcentergraphics{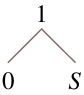} \quad \to \quad \vcentergraphics{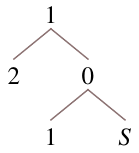}
\]
at the parent of leaf $i$, which preserves the labels of all other vertices.  If leaf $i$ is a right leaf in $T_2'$, then $T_2$ parses the word obtained by replacing $w_i$ by $12$ since now the replacement is
\[
	\vcentergraphics{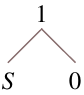} \quad \to \quad \vcentergraphics{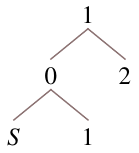}.
\]

Clearly $T_1$ parses both of these words, so we have found a parse word for the pair.  Moreover, every parse word of $T_1$ and $T_2$ arises uniquely in this way.
\end{proof}

In Section~\ref{parameterized families} we referred to the two leaves of maximal depth in a path tree as \emph{bottom leaves}.  In a general binary tree, a \emph{bottom leaf} is a leaf whose sibling is also a leaf.  It is clear that for $n \geq 2$ every $n$-leaf binary tree has at least one pair of bottom leaves, and every binary tree that is not a path tree has at least two pairs of bottom leaves.  We use these facts in the next two theorems.

\begin{theorem}\label{comb-general}
Let $n \geq 2$, and let $T$ be an $n$-leaf binary tree.  Let $l$ be the level of leaf $1$ in $T$.  Then $|\ParseWords(T, \LeftCombTree(n))| = 2^{l-1}$.
\end{theorem}

By symmetry, the analogous result holds for the right comb.

\begin{proof}
We work by induction on $n$.  The only $2$-leaf binary tree is $\LeftCombTree(2) = \smalltree{2}{1}$, which has only one parse word up to permutation of the alphabet.

Let $T$ be an $n$-leaf binary tree.  Then $T$ has a pair of bottom leaves; suppose these are leaves $i$ and $i+1$.  Remove these two leaves to obtain $T'$, which has $n-1$ leaves.  If $i = 1$, then Proposition~\ref{bottom-bottom} gives twice as many parse words for $T$ and $\LeftCombTree(n)$ as parse words for $T'$ and $\LeftCombTree(n-1)$.  If $i > 1$, then leaf $i$ is a right leaf in $\LeftCombTree(n-1)$, so Proposition~\ref{bottom-comb} gives the same number of parse words as for $T'$ and $\LeftCombTree(n-1)$.
\end{proof}

Csar, Sengupta, and Suksompong~\cite{Csar--Sengupta--Suksompong} have recently provided a generalization of Theorem~\ref{comb-general}.
They consider a partial ordering on the set of $n$-leaf binary trees arising from the rotation operation.
They show that if $T_1$ and $T_2$ are $n$-leaf binary trees whose join exists under this partial ordering, then $|\ParseWords(T_1, T_2)|$ is a certain power of $2$.

The following theorem is an extension of Theorem~\ref{comb-general} to turn trees, although we lose the enumeration.

\begin{theorem}\label{turn-general}
Let $n \geq 4$.  Let $T_1$ be an $n$-leaf binary tree and $T_2$ an $n$-leaf left turn tree.  Then $T_1$ and $T_2$ have a common parse word.
\end{theorem}

\begin{proof}
We work by induction on $n$.  For $n=4$ the result can be verified explicitly.

Now suppose that every $(n-1)$-leaf binary tree has a common parse word with every $(n-1)$-leaf left turn tree.  Let $T_1$ be an $n$-leaf binary tree, and let $T_2$ be an $n$-leaf left turn tree.  Then $T_1$ has a pair of bottom leaves; suppose these are leaves $i$ and $i+1$.

First we consider the case where the $i$th leaf of $T_2$ is the right bottom leaf.  If $T_1$ is a path tree, then $T_1$ and $T_2$ have a common parse word by Proposition~\ref{shared bottom leaf}.  If $T_1$ is not a path tree, then there is another pair of bottom leaves in $T_1$, so we may re-choose $i$ if necessary so that the $i$th leaf of $T_2$ is not the right bottom leaf.

Therefore we may assume that the $i$th leaf of $T_2$ is not the right bottom leaf.  Remove leaves $i$ and $i+1$ from $T_1$ to obtain $T_1'$, which has $n-1$ leaves and so has a common parse word with every $(n-1)$-leaf left turn tree.

If the $i$th leaf of $T_2$ is the left bottom leaf, then we can apply Proposition~\ref{bottom-bottom} to obtain a common parse word for $T_1$ and $T_2$.  Otherwise, leaves $i$ and $i+1$ occur on consecutive levels in $T_2$, so Proposition~\ref{bottom-comb} applies.
\end{proof}

\section{A pair of turn trees}\label{turn trees}

In this section we give three theorems that collectively determine the number of parse words of $\LeftTurnTree(m, n)$ and $\RightTurnTree(k, m + n - k)$.  Note that by Theorem~\ref{turn-general} the number of parse words is nonzero.

\begin{theorem}\label{unique turn-turn count}
For $m \geq 1$, $k \geq 1$, and $\max(2, k - m + 2) \leq n \leq k$,
\[
	|\ParseWords(\LeftTurnTree(m, n), \RightTurnTree(k, m + n - k))| = 1.
\]
\end{theorem}

\begin{proof}
The bottom leaves of $\LeftTurnTree(m,n)$ (which are leaves $n-1$ and $n$) correspond to leaves which are on consecutive levels in $\RightTurnTree(k, m + n - k)$, so we can apply Proposition~\ref{bottom-comb} to see that
\begin{multline*}
	|\ParseWords(\LeftTurnTree(m, n), \RightTurnTree(k, m + n - k))| \\
	= |\ParseWords(\LeftTurnTree(m, n-1), \RightTurnTree(k-1, m + n - k))|.
\end{multline*}
Now, our hypothesis applies to this new, smaller tree pair, so we may continue reducing in the same way until we have reduced the right comb in the left turn tree entirely away.  At this point, we are considering the trees $\LeftTurnTree(m, 2) = \LeftCombTree(m + 2)$ and $\RightTurnTree(k - (n-2), m + n - k)$, which have a unique parse word class by Theorem~\ref{comb-general}.
\end{proof}

Let
\[
	a(m, k) = |\ParseWords(\LeftTurnTree(m, k+1), \RightTurnTree(k, m+1))|.
\]
By considering the left--right reflections of these two trees, we see that $a(m, k) = a(k, m)$.  Theorem~\ref{turn-turn count} determines the number of parse words of $\LeftTurnTree(m, n)$ and $\RightTurnTree(k, m + n - k)$ for $n \geq k + 2$ in terms of $a(m, k)$, and Theorem~\ref{2-parameter turn-turn count} evaluates $a(m, k)$.

\begin{theorem}\label{turn-turn count}
For $m \geq 1$, $k \geq 1$, and $n \geq k + 2$,
\[
	|\ParseWords(\LeftTurnTree(m, n), \RightTurnTree(k, m + n - k))| = 2 a(m, k).
\]
\end{theorem}

\begin{proof}
If $n > k+2$, then the bottom leaves of $\LeftTurnTree(m,n)$ correspond to leaves which are on consecutive levels in $\RightTurnTree(k, m + n - k)$, so we can apply Proposition~\ref{bottom-comb} to see that
\begin{multline*}
	|\ParseWords(\LeftTurnTree(m, n), \RightTurnTree(k, m + n - k))| \\
	= |\ParseWords(\LeftTurnTree(m, n-1), \RightTurnTree(k, m + n - k-1))|.
\end{multline*}
If our hypothesis applies to this new, smaller tree pair, we may continue reducing in exactly the same way until we reach $\LeftTurnTree(m, k + 2)$ and $\RightTurnTree(k, m + 2)$.  Leaves $k$ and $k+1$ are bottom leaves in both these trees, so by Proposition~\ref{bottom-bottom} we have
\begin{multline*}
	|\ParseWords(\LeftTurnTree(m, n), \RightTurnTree(k, m + n - k))| \\
	= 2 |\ParseWords(\LeftTurnTree(m, k + 1), \RightTurnTree(k, m + 1))|. \qedhere
\end{multline*}
\end{proof}

For the final result concerning the number of parse words of two turn trees, it turns out to be convenient to focus on the labels of the internal vertices rather than of the leaves.  We form a word consisting of the internal vertex labels of a labeled path tree by reading these labels from top to bottom.

A word on $\{0,1,2\}$ is \emph{alternating} if no two consecutive letters are equal.  If the internal vertices of a path tree are labeled with $w$, then the labeling can be extended to a parse word for the tree precisely when $w$ is alternating.  Therefore it will be important to know the sizes of certain sets of alternating words.  Let $A_m$ be the set of length-$m$ alternating words of the form $0 v_2 \cdots v_m$, where $v_2, v_m \in \{1,2\}$.  Let $B_m$ be the set of length-$m$ alternating words of the form $0 v_2 \cdots v_m$, where $v_2 \in \{1,2\}$ and $v_m \in \{0,2\}$.

\begin{proposition}\label{alternating counts}
For $m \geq 2$, $|A_m| = (2^m + 2 (-1)^m)/3$ and $|B_m| = (2^m - (-1)^m)/3$.
\end{proposition}

\begin{proof}
Let $a_i(m)$ be the number of length-$m$ alternating words on $\{0,1,2\}$ beginning with $01$ and ending with $\Mod(i,3)$.  Then
\begin{align*}
	a_i(m)
	&= a_{i+1}(m-1) + a_{i+2}(m-1) \\
	&= a_{i+2}(m-2) + 2 a_{i+3}(m-2) + a_{i+4}(m-2) \\
	&= a_{i+3}(m-3) + 3 a_{i+4}(m-3) + 3 a_{i+5}(m-3) + a_{i+6}(m-3) \\
	&\;\:\vdots \\
	&= \sum_{j=0}^n \binom{n}{j} a_{i+n+j}(m-n) \\
	&\;\:\vdots \\
	&= \sum_{j=0}^{m-2} \binom{m-2}{j} a_{i+m-2+j}(2) \\
	&= \sum_{j \equiv -(i+m) \bmod 3} \binom{m-2}{j}
\end{align*}
since $a_0(2) = a_2(2) = 0$ and $a_1(2) = 1$.  Therefore
\[
	a_0(m) = \sum_{j \equiv -m \bmod 3} \binom{m-2}{j} = \frac{1}{3} \left(2^{m-2} + (-1)^{m-1}\right).
\]

Since $a_1(m)$ and $a_2(m)$ also count alternating words of the forms $02 \cdots 2$ and $02 \cdots 1$ respectively, we have
\[
	|A_m| = 2 a_1(m) + 2 a_2(m) = 2 \left(2^{m-2} - a_0(m)\right) = \frac{1}{3} \left(2^m + 2 (-1)^m\right).
\]
Similarly,
\[
	|B_m| = 2 a_0(m) + a_1(m) + a_2(m) = a_0(m) + 2^{m-2} = \frac{1}{3} \left(2^m - (-1)^m\right). \qedhere
\]
\end{proof}

Next we provide a simple recurrence satisfied by $a(m, k)$.  Unfortunately, we do not know a correspondingly simple proof.

\begin{theorem}\label{2-parameter turn-turn count}
For $m \geq 1$ and $k \geq 1$,
\[
	a(m+3, k) - 2 a(m+2, k) - a(m+1, k) + 2 a(m, k) = 0.
\]
\end{theorem}

Initial conditions that suffice to completely determine $a(m, k)$ from this recurrence are $a(1, 1) = 1$, $a(1, 2) = 1$, $a(1, 3) = 1$, $a(2, 2) = 4$, $a(2, 3) = 5$, and $a(3, 3) = 3$.  The particular solution can be written as the matrix product
\[
	a(m, k)
	= \frac{1}{4}
	\begin{pmatrix}
		2/3 \cdot 2^m \\
		1 \\
		5/3 \cdot (-1)^m
	\end{pmatrix}^\top
	\begin{pmatrix}
		1/2 & 1 & 1 \\
		1 & 1 & -1 \\
		1 & -1 & 1/5
	\end{pmatrix}
	\begin{pmatrix}
		2/3 \cdot 2^k \\
		1 \\
		5/3 \cdot (-1)^k
	\end{pmatrix}.
\]

\begin{proof}
Let $m \geq 2$ and $k \geq 2$, and let $A_m$ and $B_m$ be as above.  Let
\begin{align*}
	1_m(w) &= \left\{((01)^{m/2} w, w (10)^{m/2})\right\}, \\
	2_m(w) &= \left\{((02)^{m/2} w, w (20)^{m/2})\right\}, \\
	A_m(w) &= \left\{(v w, w v) : v \in A_m\right\}, \\
	B_m((01)^{k/2}) &= \left\{(v (10)^{k/2}, (01)^{k/2} v) : v \in B_m\right\}.
\end{align*}

We consider the set of pairs $(L, R)$ of length-$(m+k)$ (alternating) words such that $R = 01\cdots$ and such that respectively labeling the internal vertices of $\LeftTurnTree(m, k+1)$ and $\RightTurnTree(k, m+1)$ with $L$ and $R$ produces a parse word for the pair.  Such pairs $(L, R)$ are in bijection with equivalence classes of parse words for this tree pair as follows.  The internal vertex labels of a path tree determine the labels of all leaves except the bottom leaves.  Since $\LeftTurnTree(m, k+1)$ and $\RightTurnTree(k, m+1)$ do not share both bottom leaves, labeling the internal vertices with the pair $(L, R)$ determines a unique parse word.  We may choose representative parse words so that the internal vertex labels of $\RightTurnTree(k, m+1)$ begin with $01$ (since the first two labels cannot be the same).

Let $w = 01\cdots$ be the length-$k$ prefix of $R$.  Thus the internal vertices of the right comb of $\RightTurnTree(k, m+1)$ are labeled with letters from $w$, and the first letter of the parse word is $2$.  We show that if $w$ contains all three letters then the set of pairs $(L, R)$ is
\[
\begin{cases}
	\{\}
		& \text{if $w$ ends in $0$ and $m$ is odd} \\
	1_m(w) \cup 2_m(w)
		& \text{if $w$ ends in $0$ and $m$ is even} \\
	A_m(w)
		& \text{if $w$ ends in $1$ or $2$ and $m$ is odd} \\
	A_m(w) \cup 2_m(w)
		& \text{if $w$ ends in $1$ and $m$ is even} \\
	A_m(w) \cup 1_m(w)
		& \text{if $w$ ends in $2$ and $m$ is even,}
\end{cases}
\]
and if $w = (01)^{k/2} = 0101\cdots$ contains only two letters then this set is
\[
\begin{cases}
	\left\{((01)^{(m+k)/2}, (01)^{(m+k)/2})\right\}
		& \text{if $w$ ends in $0$ and $m$ is odd} \\
	1_m(w) \cup 2_m(w)
		& \text{if $w$ ends in $0$ and $m$ is even} \\
	A_m(w) \cup B_m(w)
		& \text{if $w$ ends in $1$ and $m$ is odd} \\
	A_m(w) \cup B_m(w) \cup 2_m(w)
		& \text{if $w$ ends in $1$ and $m$ is even.}
\end{cases}
\]

To see this, first suppose that $L = v w$ and $R = w v$ for some $v$.  Then $v$ begins with $0$, so $w$ does not end in $0$, and every $v \in A_m$ produces a parse word.

Next suppose that $L = v' w$ and $R = w v$ for some $v' \neq v$.  Then in fact $v$ and $v'$ differ in every position; in particular, $v$ begins with some letter $j \neq 0$.  Let $i \in \{1, 2\}$ such that $i \neq j$.  Then the final leaf receives the label $i$ since it is a child of a $0$ leaf in $\LeftTurnTree(m, k+1)$ and a child of a $j$ leaf in $\RightTurnTree(k, m+1)$.  It follows that $v = (j0)^{m/2}$ and $v' = (0j)^{m/2}$; therefore $m$ is even, and choosing $j$ to be either $1$ or $2$ produces a parse word as long as it differs from the last letter of $w$.

Finally, suppose that $L = v w'$ for some length-$k$ word $w' \neq w$.  Then $w$ and $w'$ differ in every position; in particular, $w'$ begins with $1$, and it follows that $w = (01)^{k/2}$ and $w' = (10)^{k/2}$.  If $w$ ends in $0$, then $L = R = (01)^{(m+k)/2}$, yielding the parse word $2^k 0 2^m$.  If $w$ ends in $1$, then every $v \in B_m$ produces a parse word, and $R = w v$.

Since we know the sizes of all these sets by Proposition~\ref{alternating counts}, we can enumerate the set of internal word pairs $(L, R)$ and obtain an expression for $a(m, k)$, which as expected is symmetric in $m$ and $k$.  For fixed $k$ the expression is a linear combination of $2^m$, $1$, and $(-1)^m$, so it satisfies the recurrence stated in the theorem, which can be written
\[
	(M - 2) (M - 1) (M + 1) \, a(m, k) = 0,
\]
where $M$ is the forward shift operator in the variable $m$.

When $m = 1$ or $k = 1$ one of the two trees is a comb tree, and by Theorem~\ref{comb-general} we have $a(m, k) = 1$, which one checks is also what the general expression for $a(m, k)$ gives upon setting $m = 1$ or $k = 1$.
\end{proof}

\section{Reducing a pair of trees}\label{reducing}

How might one proceed from the theorems of the previous sections to a proof that every two $n$-leaf path trees parse a common word?  Here we introduce two notions of reducibility --- ways to reduce the problem of finding a parse word for a pair of trees to finding parse words for smaller pairs --- and give some related conjectures.

\subsection{Decomposable pairs}

Recall that if $T_1$ and $T_2$ are $n$-leaf trees such that leaves $i$ and $i+1$ are siblings in both trees, then Proposition~\ref{bottom-bottom} reduces the problem of finding a parse word for $T_1$ and $T_2$ to the problem of finding a parse word for the pair of $(n-1)$-leaf trees in which the common $\smalltree{2}{1}$ has been removed.  Our first observation is that there is nothing special about $\smalltree{2}{1}$; if the two trees have any common branch system in the same position, then we can decompose the trees.  For example, the $8$-leaf trees
\[
	T_1 = \vcentergraphics{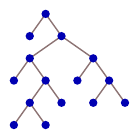} \qquad T_2 = \vcentergraphics{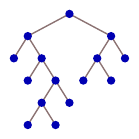}
\]
share the branch system
\[
	S = \vcentergraphics{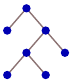}
\]
in the second through fifth leaves, which we may remove to obtain the $5$-leaf trees
\[
	\vcentergraphics{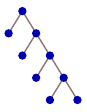} \qquad \vcentergraphics{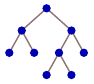}.
\]
Given a common parse word $w_1 w_2 w_3 w_4 w_5$ of this pair of $5$-leaf trees, we can find a common parse word of the original pair of $8$-leaf trees by taking any valid labeling of $S$ and permuting the alphabet so that the root receives the label $w_2$.

In fact to decompose a pair of trees we only require a vertex in $T_1$ with dangling subtree $S_1$ and a vertex in $T_2$ with dangling subtree $S_2$ such that the leaves in $S_1$ and $S_2$ are the same.  For example, there are two such vertex pairs in the tree pair
\[
	\vcentergraphics{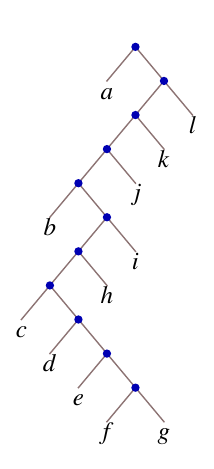} \qquad \vcentergraphics{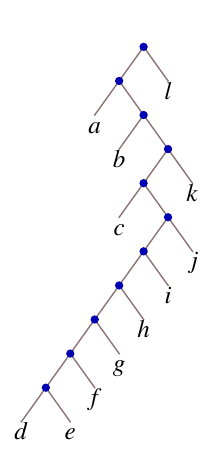}
\]
(where corresponding leaves have been given the same label).  Breaking the trees at levels $2$ and $8$ as
\[
	\vcentergraphics{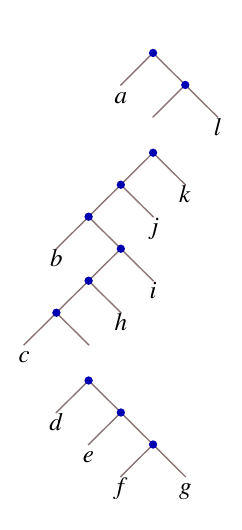} \qquad \vcentergraphics{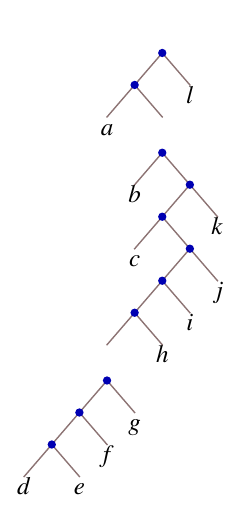}
\]
produces the same partition $\{\{a, l\}, \{b, c, h, i, j, k\}, \{d, e, f, g\}\}$ of the leaves in both trees.  Thus, to find a parse word for the original pair it suffices to find parse words for the subtree pairs.  Proposition~\ref{root label} guarantees that we can reattach the subtrees consistently, since every binary tree that parses $w$ receives the same label for its root when the leaves are labeled with the letters of $w$.  Let us call a pair of path trees \emph{indecomposable} if there is no such (nontrivial) decomposition.

The tree pair in Theorem~\ref{comb-comb} consisting of $\LeftCombTree(n)$ and $\RightCombTree(n)$ is indecomposable, as is the pair in Theorem~\ref{comb-crooked} consisting of $\LeftCombTree(n)$ and $\RightCrookedTree(n)$.  On the other hand, breaking the trees $\LeftCombTree(n)$ and $\LeftCrookedTree(n)$ at level $1$ shows that this pair is decomposable, and in this case the decomposition accounts for the non-uniqueness of the equivalence classes of words in Theorem~\ref{comb-crooked2}.

The technique of decomposing trees is not limited to path trees.  For example, the pair
\[
	T_1 = \vcentergraphics{binarytree8-69} \qquad T_2 = \vcentergraphics{binarytree8-231}
\]
can also be decomposed into the two pairs
\[
	T_1' = \vcentergraphics{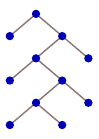} \qquad T_2' = \vcentergraphics{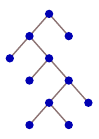}, \qquad S_1 = \vcentergraphics{binarytree3-1} \qquad S_2 = \vcentergraphics{binarytree3-2}.
\]

\subsection{Pairs that are not mutually crooked}

We showed in Proposition~\ref{bottom-comb} that if leaves $i$ and $i+1$ are siblings in $T_1$ and are on consecutive levels in $T_2$ then this pair of trees is reducible.  It is natural then to ask whether a tree pair in which leaves $i$ and $i+1$ are on consecutive levels in both trees is reducible.

First let us consider the pair
\[
	\vcentergraphics{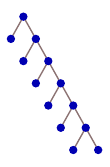} \qquad \vcentergraphics{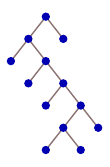}
\]
which has the common parse word $0001220$.  The three consecutive $0$s label leaves in both trees that are arranged in a right comb structure, and shortening each comb by two leaves produces the pair
\[
	\vcentergraphics{binarytree5-1} \qquad \vcentergraphics{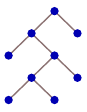}
\]
which parses $01220$.  In other words, we obtain a parse word for the larger pair by simply inserting two $0$s corresponding to the two added leaves.  Let us see why this works.  In the first tree, the $4$-leaf right comb subtree whose first three leaves are labeled $0$ has internal vertex labels $1$, $2$, and $1$; in the second tree, the corresponding $4$-leaf right comb subtree has internal vertex labels $2$, $1$, and $2$.  In both cases, the internal vertex labels alternate between $1$ and $2$, so shortening the comb by two leaves preserves the label of the root and the labels of the bottom leaves.  Hence we still have a valid labeling.

The situation will be the same even if the two combs have different orientations, and even if the trees are not path trees but binary trees in general.  We formalize this as follows.  The relevant extending operation is \emph{triplicating} leaf $i$ --- inserting two $\smalltree{2}{1}$ structures in $T'$ so as to obtain three left leaves in $T$ if leaf $i$ is a left leaf in $T'$ and three right leaves in $T$ if leaf $i$ is a right leaf in $T'$, analogous to duplicating a leaf as discussed in Section~\ref{general families}.

\begin{theorem}
Let $T_1'$ and $T_2'$ be $(n-2)$-leaf binary trees.  Let $1 \leq i \leq n-2$, and let $T_1$ and $T_2$ be the $n$-leaf trees obtained from $T_1'$ and $T_2'$ by triplicating leaf $i$.  If $w = w_1 \cdots w_{n-2}$ is a parse word for $T_1'$ and $T_2'$, then $w_1 \cdots w_{i-1} w_i w_i w_i w_{i+1} \cdots w_{n-2}$ is a parse word for $T_1$ and $T_2$.
\end{theorem}

A pair of $n$-leaf trees $T_1$ and $T_2$ is \emph{weakly mutually crooked} if it cannot be obtained by triplicating some leaf $i$ in a pair of $(n-2)$-leaf trees.  To prove that every pair of binary trees has a parse word, by the previous theorem it suffices to consider pairs of weakly mutually crooked trees.

However, it appears that something stronger is true.
A pair of $n$-leaf trees $T_1$ and $T_2$ is \emph{mutually crooked} if it cannot be obtained by duplicating some leaf $i$ in a pair of $(n-1)$-leaf trees.
That is, no pair of consecutive leaves has an uncle--nephew relationship in both trees.
For example, the pair parsing $0110212$ in Section~\ref{introduction} is mutually crooked.
Experimental evidence suggests that in fact it suffices to consider pairs of mutually crooked trees.

\begin{conjecture}
Let $T_1'$ and $T_2'$ be $(n-1)$-leaf binary trees.  Let $1 \leq i \leq n-1$, and let $T_1$ and $T_2$ be the $n$-leaf trees obtained from $T_1'$ and $T_2'$ by duplicating leaf $i$.  There exists a parse word $w = w_1 \cdots w_n$ of $T_1$ and $T_2$ such that $w_i = w_{i+1}$.
\end{conjecture}

For example, in the pair discussed at the beginning of this subsection, leaves $5$ and $6$ are on consecutive levels in both trees, and these leaves receive the label $2$.  Note however that the parse word of $T_1$ and $T_2$ is not necessarily a simple extension of a parse word of $T_1'$ and $T_2'$.

The trees $\LeftCrookedTree(n)$ and $\RightCrookedTree(n)$ (which we addressed in Theorem~\ref{crooked-crooked count}) are mutually crooked, but for $n \geq 5$ no path tree is mutually crooked to $\LeftCombTree(n)$, since even a completely crooked tree has a pair of consecutive leaves that lie in consecutive levels.  Theorems~\ref{comb-comb} and \ref{turn-turn} provide additional examples of pairs that fail to be mutually crooked.

\subsection{Other conjectures}

To prove that every pair of $n$-leaf binary trees $T_1$ and $T_2$ has a parse word, it therefore suffices to consider indecomposable, weakly mutually crooked pairs of trees.  In particular, we may assume that the leaves on level $1$ in $T_1$ and $T_2$ are different, since if they are the same then the pair is decomposable at level $1$ into smaller pairs.

\begin{theorem}
Let $n \geq 3$, and let $T_1$ and $T_2$ be $n$-leaf path trees such that leaf $1$ is on level $1$ in $T_1$ and leaf $n$ is on level $1$ in $T_2$.
Then $T_1$ and $T_2$ have no parse word of the form $01v1$.
\end{theorem}

\begin{proof}
For $n = 3$ one checks that $011$ is not a parse word for one of the two $3$-leaf binary trees.
Assume $n \geq 4$.
Toward a contradiction, suppose that $01v1$ is a common parse word for some $v$.
Then the root of each tree receives the label $2$, and the internal vertex on level $1$ of $T_1$ receives $1$.
Consider the children of this internal vertex.
If the right child is a leaf, then it is leaf $n$ and so receives $1$, which is not a valid label because its parent is already labeled $1$.
If the left child is a leaf, then it is leaf $2$ and so receives $1$, which is also not a valid label.
\end{proof}

A similar argument shows that if there is a parse word of the form $01v2$, then leaf $2$ of $T_1$ is on level $2$, and leaf $n-1$ of $T_2$ is on level $2$.

The following conjecture gives several statements that seem to be true and may be helpful in proving Theorem~\ref{main} for path trees directly.

\begin{conjecture}
Let $n \geq 4$, and let $T_1$ and $T_2$ be $n$-leaf path trees such that leaf $1$ is on level $1$ in $T_1$ and leaf $n$ is on level $1$ in $T_2$.  Then we have the following.
\begin{itemize}
\item If $T_1$ and $T_2$ have no parse word of the form $00v$ or $v00$, then they have a unique parse word (up to permutation of alphabet).
\item If $T_1$ and $T_2$ have no parse word of the form $00v$ and are mutually crooked, then they have a parse word of the form $01v00$.
\item If $T_1$ and $T_2$ have no parse word of the form $00v$, then the only possibilities for the $2$-tuple
\[
	(\text{level of leaf $2$ in $T_1$}, \text{level of leaf $n-1$ in $T_2$})
\]
are $(2,3)$ and $(k,2)$ for some $k \geq 2$. \\
Moreover, if $T_1$ and $T_2$ are weakly mutually crooked, the only possibilities are $(2,3)$ and $(k,2)$ for some $2 \leq k \leq 4$. \\
Moreover, if $T_1$ and $T_2$ are mutually crooked, the only possibilities are $(2,3)$ and $(k,2)$ for some $2 \leq k \leq 3$.
\end{itemize}
\end{conjecture}

Finally, we give an interesting conjecture that has been explicitly verified for $n \leq 12$.  (The statement does not hold for general binary trees.)

\begin{conjecture}
Let $n \geq 4$.  Every pair of $n$-leaf path trees parses a word of the form $u00v$ for some (possibly empty) $u, v$.
\end{conjecture}


\begin{thebibliography}{9}

\bibitem{Appel--Haken}
Kenneth Appel and Wolfgang Haken,
Every planar map is four colorable I: discharging,
\emph{Illinois Journal of Mathematics} \textbf{21} (1977) 429--490.

\bibitem{Appel--Haken--Koch}
Kenneth Appel, Wolfgang Haken, and John Koch,
Every planar map is four colorable II: reducibility,
\emph{Illinois Journal of Mathematics} \textbf{21} (1977) 491--567.

\bibitem{Csar--Sengupta--Suksompong}
Sebastian A. Csar, Rik Sengupta, Warut Suksompong,
On a subposet of the Tamari lattice,
available from \url{http://arxiv.org/abs/1108.5690}.

\bibitem{Kauffman}
Louis Kauffman,
Map coloring and the vector cross product,
\emph{Journal of Combinatorial Theory, Series B} \textbf{48} (1990) 145--154.

\bibitem{ParseWords}
Eric Rowland,
\textsc{ParseWords},
available from \url{http://www.cs.uwaterloo.ca/~erowland/packages.html}.

\bibitem{LOU}
Doron Zeilberger,
\textsc{LOU},
available from \url{http://www.math.rutgers.edu/~zeilberg/programs.html}.

\end{thebibliography}
\end{document}